\documentclass[american,11pt,reqno,twoside]{amsart}
\usepackage[utf8]{inputenc}
\usepackage[T1]{fontenc}
\usepackage[paper=a4paper]{geometry}
\usepackage[all]{xy}
\usepackage{amssymb,mathtools}
\usepackage{hyperref}
\usepackage{amsthm}
\usepackage{stmaryrd}
\usepackage{enumitem}
\usepackage{color,xcolor}
\usepackage{babel}
\usepackage[all]{nowidow}
\usepackage[babel]{microtype}
\DeclareMathOperator{\A}{A}
\DeclarePairedDelimiter{\accol}{\{}{\}}
\DeclareMathOperator{\B}{B}
\newcommand*{\CC}{\mathbb{C}}
\DeclareMathOperator{\CCM}{\mathsf{CM}}
\DeclarePairedDelimiter{\CCrochet}{\llbracket}{\rrbracket}
\newcommand*{\DEUX}{\mathrm{II}}
\DeclareMathOperator{\Diff}{D}
\newcommand*{\Dz}{\Diff_z}
\DeclareMathOperator{\E}{E}
\newcommand*{\ee}{\mathrm{e}}
\DeclareMathOperator{\F}{F}
\DeclareMathOperator{\FRC}{\mathsf{FRC}}
\newcommand*{\ic}{\mathrm{i}}
\DeclareMathOperator{\Id}{Id}
\DeclareMathOperator{\Ob}{Ob}
\DeclareMathOperator{\J}{J}
\newcommand*{\Mod}{\mathcal{M}}
\newcommand*{\N}{\mathbb{N}}
\DeclarePairedDelimiter{\Poch}{\langle}{\rangle}
\newcommand*{\pk}{\mathcal{H}}
\newcommand*{\QM}{\Mod^{\leq\infty}}
\DeclareMathOperator{\RC}{\mathsf{RC}}
\DeclareMathOperator{\Se}{Se}
\newcommand*{\sharpp}{\mathbin{\#}}
\DeclareMathOperator{\SRC}{\mathsf{SRC}}
\newcommand*{\sldz}{\mathrm{SL}(2,\Z)}
\newcommand*{\starp}{\mathbin{\star}}
\DeclareMathOperator{\T}{\mathsf{T}}
\newcommand*{\trH}[3]{\mathop{#1\vert_{\raisebox{-2pt}{\(\scriptscriptstyle #2\)}}}{#3}}
\newcommand*{\TROIS}{\mathrm{III}}
\newcommand*{\trJ}[3]{\mathop{#1\Vert_{\raisebox{-2pt}{\(\scriptscriptstyle #2\)}}}(#3)}
\newcommand*{\UN}{\mathrm{I}}
\newcommand*{\wJ}{\mathcal{J}}
\newcommand*{\wK}{\mathcal{K}}
\newcommand*{\wQ}{\mathcal{Q}}
\newcommand*{\Z}{\mathbb{Z}}

\newtheorem{theorem}{Theorem}[subsection]
\newtheorem{lemma}[theorem]{Lemma}
\newtheorem{proposition}[theorem]{Proposition}
\theoremstyle{definition}
\newtheorem{definition}[theorem]{Definition}
\newtheorem{remark}[theorem]{Remark}
\hypersetup{%
   unicode=true,
   hidelinks,%
   pdftitle={Formal deformations of the algebra of Jacobi forms and Rankin-Cohen brackets},%
   pdfauthor={\textcopyright\ YoungJu Choie, Fran{\c c}ois Dumas, Fran{\c c}ois Martin et Emmanuel Royer},%
   pdfsubject={Jacobi forms, Poisson brackets, Rankin-Cohen brackets, formal deformation, star product},%
   pdfkeywords={Jacobi forms, quasimodular forms, Poisson brackets, Rankin-Cohen brackets, formal deformation, star product, Connes-Moscovici Theorem, Serre derivative},%
   baseurl={http://lmbp.uca.fr/~royer/},%
   pdflang=en-US,
  }
\title[Deformations of Jacobi forms]{Formal deformations of the algebra of Jacobi forms and Rankin-Cohen brackets}
\author{YoungJu Choie}
\address[Y. ~Choie]{Pohang University of Science and Technology, Department of Mathematics, Pohang, Korea}
\email{yjc@postech.ac.kr}
\author{Fran\c cois Dumas}
\address[F.~Dumas]{Université Clermont Auvergne, CNRS, LMBP, F-63000 Clermont-Ferrand, France}
\email{Francois.Dumas@uca.fr}
\author{Fran\c cois Martin}
\address[F.~Martin]{Université Clermont Auvergne, CNRS, LMBP, F-63000 Clermont-Ferrand, France}
\email{Francois.Martin@uca.fr}
\author{Emmanuel Royer}
\address[E.~Royer]{Université Clermont Auvergne, CNRS, LMBP, F-63000 Clermont-Ferrand, France}
\email{Emmanuel.Royer@math.cnrs.fr}
\begin{document}
\begin{abstract}{This work is devoted to the algebraic and arithmetic properties of Rankin-Cohen brackets allowing to define and study them in several natural situations of number theory. It focuses on the property of these brackets to be formal deformations of the algebras on which they are defined, with related questions on restriction-extension methods. The general algebraic results developed here are applied to the study of formal deformations of the algebra of weak Jacobi forms and their relation with the Rankin-Cohen brackets on modular and quasimodular forms.}
\end{abstract}
\date\today
\thanks{The first author is funded by grant NRF 2017R1A2B2001807 of the National Research Fund of Korea. The three other authors are partially funded by the project CAP 20--25.}
\maketitle
\section*{Introduction} 

Appearing in the late 1950s in the study of modular forms, Rankin-Cohen brackets have undergone considerable development in many related fields, giving rise to a very abundant literature in recent decades. 
The initial problem was to construct bi-differential operators in two variables in such a way that their evaluation at modular forms is still a modular form. 
The preservation of this property of ${\rm SL}(2)$-equivariance was the main objective of the generalizations proposed for other algebras of functions of arithmetic origin 
(such as quasimodular forms and Jacobi forms, see for example \cite{MR2186573,MR1605899,MR1895295}) with respect to the appropriate arithmetical parameters (weight, depth, index).
It was also considered in various contexts related to Lie theory, representation theory or differential 
geometry, see for instance \cite{MR2325700, MR3477337,MR2271014,MR1967533}).
But there is a second fundamental aspect of the families of Rankin-Cohen brackets, namely the fact that they define formal associative deformations of the algebras considered and thus appear as 
alternative versions of the families of transvectants in the classical theory of invariants. This specific global property of Rankin-Cohen brackets is the focus of this paper. 
Number theory is the main framework of this work: both the motivations of the algebraization of the problem and the applications are related to modular, quasimodular and Jacobi forms.

The fact that Rankin-Cohen brackets define a formal deformation on modular forms is mentioned as a final remark by Eholzer in Zagier's article \cite{Z}. 
This fact encodes a large set of relations between the arithmetic functions build from the Fourier coefficients of the modular forms. Understanding this set of relations is indeed the very first motivation of the seminal work by Zagier~\cite{Z}  (see also \cite{MP1}). 
This property and the resulting links between Rankin-Cohen brackets and quantization procedures have given rise to many significant articles, among which we can cite 
Unterberger and Unterberger \cite{UU}, Cohen, Manin and Zagier \cite{MR1418868}, Connes and Moscovici \cite{MR2074985}, Bieliavsky, Tang and Yaol\cite{MR2319770}, with in Pevzner \cite{MP2} an enlightening perspective on their different points of view. 
It is impossible to give here complete references on such a vast subject, but it is necessary for our study to mention that the article \cite{MR2074985} gives as a corollary of more general results an explicit method to construct from any homogeneous derivation~\(D\) on a graded algebra \(A\) formal Rankin-Cohen brackets which give a deformation on \(A\); this type of brackets correspond to the notion of standard RC algebra in \cite{Z}. This general process cannot be applied directly to the algebra \(\Mod\) of modular forms because it is not stable by the complex derivative. That's why Zagier introduced in \cite{Z} (see also \cite{Z123}) a more subtle argument to define formal Rankin-Cohen brackets on the algebra \(\QM\) of quasimodular forms whose restriction to \(\Mod\) gives precisely the classical Rankin-Cohen brackets. We give in this paper a formalisation of this extension-restriction argument to the general framework of abstract differential algebras and use it to extend the classical Rankin-Cohen brackets on \(\Mod\) into a deformation of the algebra of weak Jacobi forms. 

The text is organized in five sections. The first is devoted to recalling the basic notions and results on formal deformations and to formulate some of their corollaries in terms adapted to our study; we specify in particular in Propositions~\ref{prop_FRC} and~\ref{prop_biFRC} the notion of formal Rankin-Cohen brackets associated with a homogeneous derivation of a graded or bigraded algebra.     

In the second section, we give two possible strategies to construct significant formal deformations on a graded algebra \(A\). The most direct one is to start with an homogeneous derivation \(D\) of \(A\) and to consider the associated formal Rankin-Cohen brackets. The second strategy is to embed \(A\) into an algebra \(R\) and consider a suitable derivation \(D\) of \(R\) which is not a derivation of \(A\), so that the formal Rankin-Cohen brackets on \(R\) associated to \(D\) restrict into a deformation of \(A\). This is the contents of Theorem~\ref{extrestrict}. In the particular case where \(A=\Mod\) we apply the first strategy taking for \(D\) the Serre derivative, and the second one with \(R=\QM\) and \(D\) the complex derivative. We prove in Proposition~\ref{notiso} that the two formal deformations of \(A\) obtained by these two approaches are not isomorphic. We develop these two strategies for weak Jacobi forms in the rest of the paper.

The third section gathers useful notions on the weak Jacobi forms for the full modular group, according to~\cite{MR781735}. The algebra of weak Jacobi forms \(\wJ\) is a polynomial algebra in four variables \(\E_4, \E_6, \A\) and \(\B\) over \(\CC\), bigraded by the weight and the index, containing \(\Mod=\CC [\E_4,\E_6]\) as a subalgebra. 

We define in the fourth section a family of derivations of \(\wJ\) extending the Serre derivation on \(\Mod\) from which we deduce in Theorem~\ref{firstfamily} a family parameterized by \(\CC^3\) of formal Rankin-Cohen brackets on \(\wJ\). We classify in Theorem~\ref{classification} these deformations of \(\wJ\) up to modular isomorphism.

The last section deals with the natural problem of extending the classical Rankin-Cohen brackets on \(\Mod\) to a deformation of \(\wJ\). To do this, we implement the extension-restriction method based on the Theorem~\ref{extrestrict}. Here the considered extension is not a polynomial extension as in the case of the embedding of \(\Mod\) in \(\QM\), but an extension by localization.  More precisely we introduce the Laurent polynomial algebra \(\wK=\CC[\E_4,\E_6,\A^{\pm 1},\B]\) which contains \(\wJ\) and also a copy \(\wQ=\CC[\E_4,\E_6,\F_2]\) of \(\QM\), where \(\F_2=\B\A^{-1}\) is a scalar multiple of the Weierstra\ss{} function. We define in Theorem~\ref{thm3} a family parameterized by \(\CC^2\) of formal Rankin-Cohen brackets on \(\wK\) whose restriction to \(\Mod\) are the classical Rankin-Cohen brackets on \(\Mod\), and determine in Theorem~\ref{thm4} the values of the parameters for which these brackets give deformations of \(\wJ\).

\section{Algebraic results on formal deformations}\label{firstsection}
\subsection{Formal deformations}
In this section, we recall the basic properties of formal deformations and their isomorphisms. Our main reference on this subject is~\cite[Chapter 13]{MR2906391}. 
\subsubsection{Definition and first properties}\label{eq_deffirstprop}
For any commutative \(\CC\)-algebra \(A\), let  \(A[[\hbar]]\) be the commutative algebra of formal power series in one variable \(\hbar\) with coefficients in \(A\). A \emph{formal deformation} of \(A\) is a family \((\mu_j)_{j\geq 0}\) of bilinear maps \(\mu_j\colon A\times A\to A\) such that \(\mu_0\) is the product of \(A\) and such that the (non commutative) product on \(A[[\hbar]]\) defined by extension of
\begin{equation}\label{star}
\forall (f,g)\in A^2\qquad f\starp g=\sum_{j\geq 0}\mu_j(f,g)\hbar^j
\end{equation}
is associative. This associativity is reflected in
\begin{equation}\label{associativity}
\forall n\geq 0\quad \forall(f,g,h)\in A^3\qquad\sum_{r=0}^n\mu_{n-r}(\mu_r(f,g),h)=\sum_{r=0}^n\mu_{n-r}(f,\mu_r(g,h)).
\end{equation} 
If \((\mu_j)_{j\geq 0}\) is a formal deformation of \(A\),  if \(\mu_1\) is  skew-symmetric and \(\mu_2\) is symmetric, then \(\mu_1\) is a Poisson bracket on \(A\).
\subsubsection{Isomorphic formal deformations}\label{equivfd}
Let \((\mu_j)_{j\geq 0}\) and  \((\mu'_j)_{j\geq 0}\) be two formal deformations of \(A\).  They are \emph{isomorphic} if there exists a \(\CC\)-linear bijective map \(\phi\colon A\to A\) such that
\begin{equation}\label{equivfdid}
\forall j\geq 0\quad \forall(f,g)\in A^2\qquad\phi(\mu_j(f,g))=\mu'_j(\phi(f),\phi(g)).%
\end{equation}
Assume that \(\mu_1\) is  skew-symmetric and \(\mu_2\) is symmetric. Formula~\eqref{equivfdid} for \(j=0\) and \(j=1\) implies, in particular, that \(\phi\) is an automorphism of the Poisson algebra \((A,\mu_1)\). We denote by \(\starp\) and \(\sharpp\) the products on \(A[[\hbar]]\) respectively associated to the formal deformations \((\mu_j)_{j\geq 0}\) and  \((\mu'_j)_{j\geq 0}\). The \(\CC[[\hbar]]\)-linear extension \(\phi\colon A[[\hbar]]\to A[[\hbar]]\) satisfies
\begin{equation}
\forall(f,g)\in A^2\qquad\phi(f\starp g)=\phi(f)\sharpp\phi(g). 
\end{equation}
\subsubsection{Example}\label{exfd}
It is well known that, if \(d\)  and \(\delta\) are two \(\CC\)-derivations of \(A\) satisfying \(\delta d=d\delta\), then the sequence \((\T_n^{d,\delta})_{n\geq 0}\) of formal transvectants \(\T_n^{d,\delta}\colon A\times A\to A\) defined 
for any \(f,g\in A\) by
\begin{equation}\label{defT}
\T_n^{d,\delta}(f,g)=\sum_{r=0}^{n}\frac{(-1)^r}{r!(n-r)!}d^{n-r}\delta^{r}(f)d^{r}\delta^{n-r}(g)
\end{equation} 
is a formal deformation of \(A\). The next paragraph is devoted to the more complicated situation where the two derivations don't commute but generate the two dimensional non abelian Lie algebra.
\subsection{Connes-Moscovici deformations}\label{CM}
We recall in the following proposition a particular case of a theorem due to Connes \& Moscovici which provides a general method for constructing formal deformations. 
\begin{definition}\label{def_CM}
Let \(A\) a commutative \(\CC\)-algebra, and \(\Delta\) and \(D\) two \(\CC\)-derivations of \(A\) satisfying 
\begin{equation}\Delta D-D\Delta=D.\end{equation}
The Connes-Moscovici deformation on \(A\) associated to \((D,\Delta)\) is the sequence  \((\CCM_n^{D,\Delta})_{n\geq 0}\) of bilinear maps \(A\times A\to A\) defined for any \(f,g\in A\) by
\begin{equation}\label{defCM}
\CCM_n^{D,\Delta}(f,g)= \sum_{r=0}^{n}\frac{(-1)^r}{r!(n-r)!}D^{r}(2\Delta+r)^{\Poch{n-r}}(f)D^{n-r}(2\Delta+n-r)^{\Poch{r}}(g),
\end{equation}
with convention \(1=\Id_A\) and for any function \(F\colon A\to A\) the Pochhammer notation:
\begin{equation}\label{poch}
F^{\Poch{0}}=1  \quad\text{and}\quad F^{\Poch{m}}=F\,(F+1)\,\cdots(F+m-1) \ \text{ for any \(m\geq 1\)}.
\end{equation}
\end{definition}
\begin{proposition}\label{prop_CM}
Let \(D\) and \(\Delta\) be two derivations on \(A\) such that \(\Delta D-D\Delta=D\).
Then, \((\CCM_n^{D,\Delta})_{n\geq 0}\) is a formal deformation of \(A\). 
\end{proposition}
\begin{proof}
See \cite[eq. (1.5)]{MR2074985}, also \cite{Yao} and \cite{MR2319770}.
\end{proof}  
The relationship between Connes-Moscovici deformations and transvectants has been examined in different works, see for example \cite [Section II.2.C]{Yao}, \cite[Section 3]{MR2319770}, \cite{Ovs} and \cite{OMMY}.
\subsection{Formal Rankin-Cohen brackets}\label{FRC}
Applying the above general result to graded situations gives rise to the following notion of formal Rankin-Cohen brackets.
\begin{proposition}\label{prop_FRC}
Let \(A=\bigoplus_{k\in \N} A_k\) be a graded commutative \(\CC\)-algebra, and \(D\) a derivation of \(A\) such that \(D(A_k)\subset A_{k+2}\) for any  \(k\geq 0\). Let us consider the sequence \((\FRC^D_n)_{n\geq 0}\) of bilinear maps \(A\times A\to A\) defined by bilinear extension of
 \begin{equation}\label{FRCB}
 \FRC^D_n(f,g)=\sum_{r=0}^n(-1)^r\begin{binom}{k+n-1}{n-r}
\end{binom}\begin{binom}{\ell+n-1}{r}
\end{binom}D^r(f)D^{n-r}(g),
\end{equation}
for any \(f\in A_k, \ g\in A_\ell\). Then,
\begin{enumerate}[label=\textnormal{(\roman*)}]
 \item\label{item_un} \((\FRC_n^{D})_{n\geq 0}\) is a formal deformation of \(A\).
 \item\label{item_deux}  \(\FRC^D_n(A_k,A_\ell)\subset A_{k+\ell+2n}\) for all \(n,k,\ell\geq 0\).
 \item\label{item_trois} The associated Poisson bracket is defined by bilinear extension of
\[\FRC^D_1(f,g)=kfD(g)-\ell gD(f) \text{ for all \(f\in A_k,g\in A_\ell\).}\]
\end{enumerate}
\end{proposition}
\begin{proof}
The linear map \(\Delta\colon A\to A\) defined on each homogeneous component by \(\Delta(f)=\frac12kf\) for any \(f\in A_k\) is a derivation of \(A\) which satisfies \(\Delta D-D\Delta=D\). We compute \((2\Delta+r)^{\Poch{n-r}}(f)=\frac{(k+n-1)!}{(k+r-1)!}f\) and \((2\Delta+n-r)^{\Poch{r}}(g)=\frac{(\ell+n-1)!}{(\ell+n-r-1)!}g\).
So
\begin{equation}\label{proofFRC}
\forall n\geq 0,\quad\CCM^{D,\Delta}_n=\FRC^D_n     
\end{equation}
and \ref{item_un} follows from Proposition \ref{prop_CM}. Points~\ref{item_deux} and~\ref{item_trois} are obvious by construction.
\end{proof}

\begin{definition}\label{def_FRC}
The deformation \((\FRC^D_n)_{n\geq 0}\) of \(A\) defined in Proposition~\ref{prop_FRC} is called the sequence of formal Rankin-Cohen brackets on \(A\) associated to \(D\).
\end{definition}

The construction corresponding to Definition~\ref{def_FRC} appears in \cite[Section 5]{Z}; the same article mentions (in a note added in proof) a remark by Eholzer on the fact that it defines a formal deformation. We will need in Section 3 the following result, which is a parameterized version of the Proposition~\ref{prop_FRC} for bigraded algebras.

\begin{proposition}\label{prop_biFRC}
Let \(A=\bigoplus_{k,p\in\Z} A_{k,p}\) be a bigraded commutative \(\CC\)-algebra, and \(D\) a derivation of \(A\) such that \(D(A_{k,p})\subset A_{k+2,p}\) for any \(k,p\in\Z\). For any \(\mu\in\CC\), let us consider the sequence \((\FRC^{D,\mu}_n)_{n\geq 0}\) of bilinear maps \(A\times A\to A\) defined by bilinear extension of
 \begin{equation}\label{biFRCB}
 \FRC^{D,\mu}_n(f,g)=\sum_{r=0}^n(-1)^r\begin{binom}{k+\mu p+n-1}{n-r}
\end{binom}\begin{binom}{\ell+\mu q+n-1}{r}
\end{binom}D^r(f)D^{n-r}(g),
\end{equation}
for any \(f\in A_{k,p}, \ g\in A_{\ell,q}\). Then,
\begin{enumerate}[label=\textnormal{(\roman*)}]
\item \((\FRC_n^{D,\mu})_{n\geq 0}\) is a formal deformation of \(A\).
\item \(\FRC^{D,\mu}_n(A_{k,p},A_{\ell,q})\subset A_{k+\ell+2n,p+q}\) for all \(n\in \N\), \(k,\ell,p,q\in\Z\).
\item The associated Poisson bracket is defined by bilinear extension of
\[
 \FRC^{D,\mu}_1(f,g)=(k+\mu p)fD(g)-(\ell+\mu q) gD(f) \text{ for all \(f\in A_{k,p},g\in A_{\ell,q}\).}
 \]
 \end{enumerate}
 \end{proposition}
\begin{proof} 
For any \(\mu\in\CC\), we introduce the weighted Euler derivation \(\Delta_\mu\) of \(A\) defined by linear extension of
\(\Delta_\mu(f)=\frac12(k+\mu p)f\) for all \((k,p)\in\Z^2\) and \(f\in A_{k,p}\). It satisfies \(\Delta_\mu D-D \Delta_\mu= D\). Next, we apply Proposition~\ref{prop_CM} and similar calculations to those in the proof of Proposition~\ref{prop_FRC} 
prove that \(\CCM^{D,\Delta_\mu}_n=\FRC^{D,\mu}_n  \) and give the results.
\end{proof}
\section{An extension-restriction method on formal deformations and Rankin-Cohen brackets on modular forms revisited}

\subsection{Original problem}\label{originalpb}
The classical Rankin-Cohen brackets were originally developed for modular forms and have since been used in a wide range of literature and applications. We refer for example to~\cite{Z123} as a primary reference. Recall that, if \(f\) and \(g\) are modular forms of respective weights \(k\) and \(\ell\), the differential polynomial
\begin{equation}\label{classicalRC}
\RC_n(f,g)=\sum_{r=0}^n(-1)^r\begin{binom}{k+n-1}{n-r}
\end{binom}\begin{binom}{\ell+n-1}{r}
\end{binom}\Dz^r(f)\Dz^{n-r}(g)
\end{equation}
is a modular form of weight \(k+\ell+2n\). Here \(\Dz\) is the usual normalized derivation \(\frac{1}{2i\pi}\partial_z\) related to the complex variable \(z\).

It is important to observe that, since the algebra \(\Mod\) of modular forms is not stable by this derivation (the derivative of a modular form is not a modular form), the property of the sequence of classical Rankin-Cohen brackets defined by~\eqref{classicalRC} to be a formal deformation of \(\Mod\) cannot be obtained by direct application of Proposition~\ref{prop_FRC}.
Zagier has developed in \cite{Z} an argument to overcome this difficulty (see Paragraph~\ref{thecaseMF} below). We extend it in the following theorem to the general formal framework of Connes-Moscovici deformations.

\subsection{A general argument about the restriction of some formal deformations}
Our goal is to prove the following theorem.
\begin{theorem}\label{extrestrict}
Consider a commutative \(\CC\)-algebra \(R\) and a subalgebra \(A\) of \(R\). Let \(\Delta\) and \(\theta\) be two \(\CC\)-derivations of \(R\) such that \(\Delta\theta-\theta\Delta=\theta\).
We assume that
\begin{enumerate}[label=\textnormal{(\roman*)}]
\item \(\Delta(A)\subseteq A\) and \(\theta(A)\subseteq A\);
\item there exists \(h\in A\) such as \(\Delta(h)=2h\);
\item there exists \(x\in R, x \notin A\) such that \(\Delta(x)=x\) and \(\theta(x)=-x^2+h\).
\end{enumerate}
Then, the derivation \(D\coloneqq\theta+2x\Delta\) of \(R\) satisfies \(\Delta D-D\Delta=D\) and the Connes-Moscovici deformation \((\CCM_n^{D,\Delta})_{n\geq 0}\) of \(R\) defines by restriction to \(A\) a formal deformation of \(A\). \end{theorem}

An obvious calculation shows that \(\Delta D-D\Delta=D\). We consider in \(R\) the Connes-Moscovici deformation \((\CCM_n^{D,\Delta})_{n\geq 0}\) defined by~\eqref{defCM}. The proof of the theorem is based on the following two lemmas.

\begin{lemma}\label{lemma_itertheta}
The assumptions and notations are those of Theorem~\ref{extrestrict}. We introduce for any integer \(n\geq 0\) the linear map \(\theta^{[n]}\colon R\to R\) defined by
\begin{equation}\label{deftheta}
\theta^{[n]}=\sum_{\ell=0}^n(-1)^{n-\ell}\binom{n}{\ell}x^{n-\ell}D^\ell(2\Delta+\ell)^{\Poch{n-\ell}}.
\end{equation} 
Then, for all \(n\geq 1\), we have,
\begin{enumerate}[label=\textnormal{(\roman*)}]
\item\label{item_uno} \(\theta^{[n+1]}=\theta\theta^{[n]}+nh\theta^{[n-1]}(2\Delta+n-1)\).
\item\label{item_due} \(\theta^{[n]}(A)\subseteq A\).
\end{enumerate}
\end{lemma}

\begin{proof}
\parindent=0cm
We directly check that \(\theta^{[1]}=\theta\) and \(\theta^{[2]}=\theta^2+2h\Delta\) which shows \ref{item_uno} for \(n=1\).
We then proceed by induction. It follows from \(\theta(x)=-x^2+h\) that
\[\theta\theta^{[n]}=\UN+\DEUX+h\cdot\TROIS\]
with notations
\begin{center}
{\renewcommand{\arraystretch}{1.6}
\(\left\{\begin{matrix}
\UN=\sum_{\ell}(-1)^{n+1-\ell}\binom{n}{\ell}(n-\ell)x^{n-\ell+1}D^\ell(2\Delta+\ell)^{\Poch{n-\ell}},\hfill \\
\DEUX=\sum_{\ell}(-1)^{n-\ell}\binom{n}{\ell}x^{n-\ell}\theta D^\ell(2\Delta+\ell)^{\Poch{n-\ell}},\hfill\\
\TROIS=\sum_{\ell}(-1)^{n-\ell}\binom{n}{\ell}(n-\ell)x^{n-\ell-1}D^\ell(2\Delta+\ell)^{\Poch{n-\ell}}.\hfill
\end{matrix}\right.\)}
\end{center}
We have
\[\TROIS=-n\sum_{\ell}(-1)^{n-1-\ell}\binom{n-1}{\ell}x^{n-1-\ell}D^\ell(2\Delta+\ell)^{\Poch{n-1-\ell}}(2\Delta+n-1)\]
so
\[\theta\theta^{[n]}+nh\theta^{[n-1]}(2\Delta+n-1)=\UN+\DEUX.\]
We replace \(\theta\) by \(D-2x\Delta\) in \(\DEUX\) to obtain \(\DEUX=\DEUX_1+\DEUX_2\) with
\begin{center}
{\renewcommand{\arraystretch}{1.6}
\(\left\{\begin{matrix}
\DEUX_1=-\sum_{\ell}(-1)^{n-\ell}\binom{n}{\ell-1}x^{n+1-\ell}D^\ell(2\Delta+\ell-1)^{\Poch{n+1-\ell}}\hfill \\
\DEUX_2=-2\sum_{\ell}(-1)^{n-\ell}\binom{n}{\ell}x^{n+1-\ell}\Delta D^\ell(2\Delta+\ell)^{\Poch{n-\ell}}.\hfill
\end{matrix}\right.\)}
\end{center}
We replace \(\Delta D^\ell\) by \(D^\ell\Delta+\ell D^\ell\) to write \(\DEUX_2=\DEUX_{21}+\DEUX_{22}\) with %
\begin{center}
{\renewcommand{\arraystretch}{1.6}
\(\left\{\begin{matrix}
\DEUX_{22}=-2\sum_{\ell}(-1)^{n-\ell}\binom{n}{\ell}\ell x^{n+1-\ell}D^\ell(2\Delta+\ell)^{\Poch{n-\ell}}\hfill\\
\DEUX_{21}=-2\sum_{\ell}(-1)^{n-\ell}\binom{n}{\ell}x^{n+1-\ell}D^\ell\Delta(2\Delta+\ell)^{\Poch{n-\ell}}.\hfill
\end{matrix}\right.\)}
\end{center}
By \(2\Delta=2\Delta+(\ell-1)-(\ell-1)\) we finally find \(\DEUX_{21}=\DEUX_{211}+\DEUX_{212}\) with%
\begin{center}
{\renewcommand{\arraystretch}{1.6}
\(\left\{\begin{matrix}
\DEUX_{211}=-\sum_{\ell}(-1)^{n-\ell}\binom{n}{\ell}x^{n+1-\ell}D^\ell(2\Delta+\ell-1)^{\Poch{n+1-\ell}}\hfill\\
\DEUX_{212}=\sum_{\ell}(-1)^{n-\ell}\binom{n}{\ell}(\ell-1)x^{n+1-\ell}D^\ell(2\Delta+\ell)^{\Poch{n-\ell}}.\hfill
\end{matrix}\right.\)}
\end{center}
So we obtain
\[%
\DEUX_1+\DEUX_{211}=\sum_{\ell}(-1)^{n+1-\ell}\binom{n+1}{\ell}x^{n+1-\ell}D^\ell(2\Delta+\ell-1)^{\Poch{n+1-\ell}}
\]
and
\[
\UN+\DEUX_{22}+\DEUX_{212}=\sum_{\ell}(-1)^{n+1-\ell}\binom{n+1}{\ell}(n+1-\ell)x^{n+1-\ell}D^\ell(2\Delta+\ell)^{\Poch{n-\ell}}.
\]
Then, we compute
\begin{align*}
(2\Delta&+\ell-1)^{\Poch{n+1-\ell}}+(n+1-\ell)(2\Delta+\ell)^{\Poch{n-\ell}}\\
&= (2\Delta+\ell-1)(2\Delta+\ell)^{\Poch{n-\ell}}+(n+1-\ell)(2\Delta+\ell)^{\Poch{n-\ell}}\\
&=(2\Delta+n)(2\Delta+\ell)^{\Poch{n-\ell}}\\
&=(2\Delta+\ell)^{\Poch{n+1-\ell}}.
\end{align*}
We conclude%
\[%
\UN+\DEUX=\sum_{\ell}(-1)^{n+1-\ell}\binom{n+1}{\ell}x^{n+1-\ell}D^\ell(2\Delta+\ell)^{\Poch{n+1-\ell}},
\]
which completes the proof of~\ref{item_uno} of the lemma. As a consequence, we get~\ref{item_due}.
\end{proof}

\begin{lemma}\label{equalitylemma}
The assumptions and notations are those of Theorem~\ref{extrestrict}. We introduce, for any integer \(n\geq 0\), the bilinear map \(\Theta_n: R\times R\to R\) defined by
\begin{equation}\label{eq_TCM}%
\Theta_n(f,g)=\sum_{r=0}^n\frac{(-1)^r}{r!(n-r)!}\left(\theta^{[r]}(2\Delta+r)^{\Poch{n-r}}\right)(f)\left(\theta^{[n-r]}(2\Delta+n-r)^{\Poch{r}}\right)(g)
\end{equation}
for all \((f,g)\in R^2\). Then, we have \((\CCM_n^{D,\Delta})_{n\geq 0}=(\Theta_n)_{n\geq 0}\)
.\end{lemma}

\begin{proof}
We express \(\Theta_n\) depending on  \(D\) using ~\eqref{deftheta}, \eqref{eq_TCM} and Lemma~\ref {lemma_itertheta}. We find%
\begin{multline*}%
\Theta_n(f,g)=\sum_{r,\ell,t}\frac{(-1)^{n+r-t-\ell}}{r!(n-r)!}\binom{r}{\ell}\binom{n-r}{t}x^{n-t-\ell}\\\times D^\ell(2\Delta+\ell)^{\Poch{r-\ell}}(2\Delta+r)^{\Poch{n-r}}(f)D^t(2\Delta+t)^{\Poch{n-r-t}}(2\Delta+n-r)^{\Poch{r}}(g).
\end{multline*}%
Using \((2\Delta+\ell)^{\Poch{r-\ell}}(2\Delta+r)^{\Poch{n-r}}=(2\Delta+\ell)^{\Poch{n-\ell}}\)and \((2\Delta+t)^{\Poch{n-r-t}}(2\Delta+n-r)^{\Poch{r}}=(2\Delta+t)^{\Poch{n-t}}\),   
we obtain%
\begin{multline*}%
\Theta_n(f,g)=\\\sum_{\ell,t}\frac{(-1)^{n-t}}{(n-t-\ell)!\ell! t!}D^\ell(2\Delta+\ell)^{\Poch{n-\ell}}(f)D^{t}(2\Delta+t)^{\Poch{n-t}}(g)x^{n-t-\ell}\sum_{r}(-1)^{r-\ell}\binom{n-t-\ell}{r-\ell}.
\end{multline*}
The inner sum indexed by  \(r\) is equal to \((1-1)^{n-t-\ell}\), so we obtain%
\[%
\Theta_n(f,g)=\sum_{\ell}\frac{(-1)^\ell}{\ell!(n-\ell)!}D^\ell(2\Delta+\ell)^{\Poch{n-\ell}}(f)D^{n-\ell}(2\Delta+n-\ell)^{\Poch{\ell}}(g)=\CCM_n^{D,\Delta}(f,g),
\] and the proof is complete.
\end{proof}

\begin{proof}[\proofname{} of Theorem~\ref{extrestrict}]
We know from Proposition~\ref{prop_CM} that \((\CCM_n^{D,\Delta})_{n\geq 0}\) is a formal deformation of \(R\) but it is not clear that \(\CCM_n^{D,\Delta}(f,g)\in A\) for \(f,g\in A\). The sequence \((\Theta_n)_{n\geq 0}\) satisfies, by construction, that \(\Theta_n(f,g)\in A\) for \(f,g\in A\) but it is not clear from its definition that it is a formal deformation of \(R\). Thus, Theorem~\ref{extrestrict} follows from the equality \(\CCM_n^{D,\Delta}=\Theta_n\) proved in Lemma~\ref{equalitylemma}. Let us observe that the subalgebra \(A\) of \(R\) is stable by \(\Delta\), is not necessarily stable by \(D\), but is stable by any bilinear application  \(\CCM_n^{D,\Delta}\).
\end{proof}

\subsubsection{Extension-restriction method}\label{ERMethod}
Theorem~\ref{extrestrict} is formulated in terms of restriction from an algebra \(R\) to a subalgebra \(A\). In the following number-theoretical applications, it will be applied in terms of extension from \(A\) to \(R\) either by polynomial extension (see paragraph~\ref{fmftqmf} for modular forms) or by localization (see Section~\ref{secondRC} for weak Jacobi forms).

\subsection{Application to modular forms}\label{thecaseMF}
A reference for details on this section is~\cite{MR2186573}.
\subsubsection{Reminder and notations on modular and quasimodular forms}
It is well known that the \(\CC\)-algebra of holomorphic modular forms \(\Mod\) associated to the full modular group \(\sldz\) is the weight-graded polynomial algebra 
\begin{equation}\label{M*}
\Mod=\CC[\E_4,\E_6]=\bigoplus_{k\in 2\Z_{\geq 0},\,k\not=2}\!\!\!\!\Mod_k \quad \text{with } \  \Mod_k=\!\!\!\bigoplus_{4i+6j=k}\CC\E_4^i\E_6^j
\end{equation}
where \(\E_4\) and \(\E_6\) are the Eisenstein series of respective weights 4 and 6. The Eisenstein series \(\E_2\) is not a modular form but a quasimodular form (of weight 2 and depth 1) and the algebra \(\QM\) of quasimodular forms can be described as the polynomial algebra
\begin{equation}\label{QM_def}
\QM=\Mod[\E_2]=\CC[\E_4,\E_6,\E_2]=\bigoplus_{k\in 2\Z_{\geq 0}}\bigoplus_{s=0}^{k/2}\Mod_{k-2s}\E_2^s
\end{equation}
graded by the weight \(k\) and filtered by the depth \(s\) (corresponding to the degree in \(\E_2\)). 

We denote by \(\Dz\) the normalized complex derivative \(\Dz=\frac1{2i\pi}\partial_z\). Ramanujan relations are
\begin{equation}\label{Ramanujan}
\Dz(\E_4)=-\frac{1}{3}(\E_6-\E_4\E_2),\quad\Dz(\E_6)=-\frac{1}{2}(\E_4^2-\E_6\E_2),\quad\Dz(\E_2)=-\frac{1}{12}(\E_4-\E_2^2)
\end{equation}
In particular the subalgebra \(\Mod\) of \(\QM\) is not stable by the derivation \(\Dz\) of \(\QM\). We introduce the Serre derivative, which is the derivation \(\Se\) of \(\QM\) defined by linear extension of
\begin{equation}\label{serredef}
\Se(f)=\Dz(f)-\frac k{12}\E_2f \hspace{0.5cm} \text{for any \(f\in\QM\) of weight \(k\).}
\end{equation}
We have
\begin{equation}\label{Serrel}
\Se(\E_4)=-\frac13\E_6, \quad \Se(\E_6)=-\frac12\E_4^2,\quad \Se(\E_2)=-\frac1{12}(\E_2^2+\E_4).
\end{equation}
In particular \(\Mod\) is stable by \(\Se\) and the restriction of \(\Se\) to \(\Mod=\CC[\E_4, \E_6]\) is the derivative \(\theta=-\frac13\E_6\partial_{\E_4}-\frac12\E_4^2\partial_{\E_6}\).

\subsubsection{Application of the extension-restriction method to modular forms}\label{fmftqmf}
We apply here the general result of Theorem~\ref{extrestrict} to give another proof of the following well-known result. 
\begin{proposition}
 The sequence of classical Rankin-Cohen brackets \eqref{classicalRC} defines a formal deformation of the algebra \(\Mod\).
\end{proposition}
\begin{proof}
We choose for \(A\) the algebra of modular forms \(\Mod=\CC[\E_4,\E_6]\), for \(\theta\) the restriction to \(\Mod\) of the Serre derivative \(\Se\) and for \(\Delta\) the weight-derivative defined by \(\Delta(f)=\frac{k}{2}f\) for any \(f\in \Mod_k\). We have \(\Delta\theta-\theta\Delta=\theta\). We introduce%
\[%
 h=-\frac{1}{144}\E_4\in\Mod \quad\text{and}\quad x=\frac{1}{12}\E_2\in\QM.
\]
We consider the polynomial extension \(R=\Mod[x]\), which is, by \eqref{QM_def}, the algebra of quasimodular forms \(\QM\).  We extend \(\Delta\) and \(\theta\) to \(\QM\) by \(\Delta(x)=x\) and \(\theta(x)=-x^2+h\). In other words %
\[%
\Delta(\E_2)=\E_2 \ \text{ and } \
\theta\left(\E_2\right)=-\frac{1}{12}\E_2^2-\frac{1}{12}\E_4.
\]
Then, the extensions \(\Delta\) and \(\theta\) are the derivations of \(\QM\) such that
\[\Delta(f)=\frac{k}{2}f \text{ for any } f\in \QM \text{ of weight } k,\ \text{and}\ \theta=\Se \text{ on } \QM,\]
and they also satisfy \(\Delta\theta-\theta\Delta=\theta\). By \eqref{serredef} the derivation \(\theta+2x\Delta\) of \(\QM\) is equal to the derivative \(\Dz\). Applying Theorem~\ref{extrestrict}, we deduce that the formal deformation \((\CCM^{\Dz,\Delta})_{n\geq 0}\) of \(\QM\) defines by restriction a formal deformation of \(\Mod\). On the one hand \(\CCM^{\Dz,\Delta}_{n}=\FRC^{\Dz}_n\) on \(\QM\) by~\eqref{proofFRC}. On the other hand, the restriction of \(\FRC^{\Dz}_n\) to \(\Mod\) is the classical Rankin-Cohen bracket \(\RC_n\) defined by \eqref{classicalRC}. So we get a proof of the property that \((\RC_n)_{n\geq 0}\) is a formal deformation of \(\Mod\).
\end{proof}

As observed at the end of the proof of Theorem~\ref{extrestrict}, the algebra \(\Mod\) is stable by \(\Delta\), is not stable by \(\Dz\), but is stable by any Rankin-Cohen brackets \(\FRC^{\Dz}_n=\CCM_n^{\Dz,\Delta}\).

\subsubsection{Serre Rankin-Cohen brackets}\label{SeRCb}
Another strategy to overcome the fact that \(\Mod\) is not stable by the derivation \(\Dz\) is to change the derivation and apply directly Proposition~\ref{prop_FRC} to a weight 2 homogeneous derivation of \(\Mod\), for instance the Serre derivation \(\Se\) on \(\Mod\). So we define with~\eqref{FRCB} the Serre-Rankin-Cohen brackets \((\SRC_n)_{n\geq 0}=(\FRC_n^{\Se})_{n\geq 0}\) on \(\Mod\) by bilinear extension of
 \begin{equation}\label{SRCFM}
 \SRC_n(f,g)=\sum_{r=0}^n(-1)^r\begin{binom}{k+n-1}{n-r}
\end{binom}\begin{binom}{\ell+n-1}{r}\end{binom}\Se^r(f)\Se^{n-r}(g)
\end{equation} 
for \((f,g)\in\Mod_k\times\Mod_\ell\).
This is by Proposition~\ref{prop_FRC} a formal deformation of \(\Mod\) satisfying 
\begin{equation}
\forall n\geq 0\quad\SRC_n(\Mod_k,\Mod_\ell )\subset\Mod_{k+\ell+2n}.
\end{equation}
It follows from the definition that \(\SRC_0=\RC_0\) and \(\SRC_1=\RC_1\) on \(\Mod\). The following proposition clarifies the relationship between the Serre-Rankin-Cohen brackets and the usual Rankin-Cohen brackets~\eqref{classicalRC}.

\begin{proposition}\label{notiso}
The two formal deformations \((\RC_n)_{n\geq 0}\) and \((\SRC_n)_{n\geq 0}\) of \(\Mod\) are not isomorphic. 
\end{proposition}
\begin{proof}
If \(\varphi\) is an isomorphism between the formal deformations \(\left(\Mod,\left(\SRC\right)_{n\geq 0}\right)\) and \(\left(\Mod,\left(\RC\right)_{n\geq 0}\right)\) then it is a Poisson automorphism of \((\Mod,\RC_1)\) since \(\SRC_1=\RC_1\). Then \(\varphi=\text{id}_{\Mod}$ by Proposition 7 of~\cite{DR}. This is contradicted by
\begin{equation}\label{formuleFM}
\SRC_2(f,g)=\RC_2(f,g)+\frac1{288}k\ell(k+\ell+2)fg\E_4 \quad \text{for }(f,g)\in\Mod_k\times\Mod_\ell.
\end{equation}
\end{proof}

\subsubsection{Formal deformations on quasimodular forms}%
Using \eqref{Serrel}, the formal Rankin-Cohen brackets \((\SRC_n)_{n\geq 0}\) on \(\Mod\) defined by \eqref{SRCFM} extend canonically to \(\QM\). Many other formal deformations of \(\QM\) can be constructed using Proposition~\ref{prop_CM} and the systematic study of such analogues of Rankin-Cohen brackets on \(\QM\) is the subject of the article \cite{DR} under the additional arithmetical assumption of depth conservation. It is not satisfied by \((\SRC_n)_{n\geq 0}\) because for instance \(\SRC_1(\E_2,\E_4)=\frac13(\E_4\E_2^2-2\E_6\E_2+\E_4^2)\) is of depth 2. However, the deformations of \(\QM\) appearing in Theorems B and D of \cite{DR} will play some role in the study of the deformations on the Jacobi forms in the following section.
\section{The algebra of weak Jacobi forms}

\subsection{The notion of weak Jacobi form}\label{sec_notionofwJf}
The aim of this part is to gather the notions we shall need on weak Jacobi forms. The main reference is~\cite{MR781735}.
Let \(\pk\) be the upper half plane, \(k\) an integer and \(m\) a nonnegative integer. The Jacobi group \(\sldz^J=\sldz\ltimes\Z^2\) acts on functions \(f\colon\pk\times\CC\to\CC\) as follows: if \(\gamma=\bigl(\begin{smallmatrix}a & b\\c & d\end{smallmatrix}\bigr)\in\sldz\) and \((\lambda,\mu)\in\Z^2\), then%
\begin{multline*}%
\trJ{f}{k,m}{\left(\gamma,(\lambda,\mu)\right)}(\tau,z)=\\(c\tau+d)^{-k}\exp\!\left(2\pi\ic m\left(-\frac{c(z+\lambda\tau+\mu)^2}{c\tau+d}+\lambda^2\tau+2\lambda z+\lambda\mu\right)\right)f\left(\frac{a\tau+b}{c\tau+d},\frac{z+\lambda\tau+\mu}{c\tau+d}\right)
\end{multline*}
for all \((\tau,z)\in\pk\times\CC\). 

A weak Jacobi form of weight \(k\) and index \(m\) is a holomorphic function \(\Phi\colon\pk\times\CC\to\CC\) invariant by the action \(\mathop{\Vert_{\raisebox{-2pt}{\(\scriptscriptstyle k,m\)}}}\) of the Jacobi group and whose Fourier expansion has the shape %
\begin{equation}\label{eq_fourier}%
\Phi(\tau,z)=\sum_{n=0}^{+\infty}\sum_{\substack{r\in\Z\\ r^2\leq 4nm+m^2}}c(n,r)\ee^{2\pi\ic(n\tau+rz)}.
\end{equation}
The vector space \(\wJ_{k,m}\) of such functions is finite dimensional. We identify functions on \(\pk\times\CC\) that are not depending on the second variable with functions on \(\pk\). On this subspace, the action \(\mathop{\Vert_{\raisebox{-2pt}{\(\scriptscriptstyle k,0\)}}}\) of \(\sldz^J\) induces the classical modular action of \(\sldz\) denoted by \(\trH{}{k}{}\). The space \(\wJ_{k,0}=\Mod_k\) is the space of holomorphic modular forms of weight \(k\) on \(\sldz\) defined in~\ref{thecaseMF}. 

 The principal object of our study is the bigraded algebra of weak Jacobi forms%
\begin{equation}\label{bigradfwJ}\wJ=\bigoplus_{\substack{k\in 2\Z\\ m\geq 0}}\wJ_{k,m}.\end{equation}

\subsection{Generators  of the algebra of the weak Jacobi forms}\label{eisen}
The algebra \(\wJ\) is a polynomial algebra on two generators over the algebra \(\Mod\) of modular forms. We describe these two generators.

For \(m=0\), we already mentioned in \eqref{M*} that the Eisenstein series \(\E_4\) and \(\E_6\) generate the algebra of modular forms: \(\Mod=\CC[\E_4,\E_6]\). If \(m \neq 0\), the Eisenstein series \(\E_{k, m}\) of \(\wJ_{k, m}\) of even weight \(k \geq 4\) and index \(m\) is
\[\E_{k,m}(\tau,z)=\frac{1}{2}\sum_{\substack{(c,d)\in\Z^2\\(c,d)=1}}\sum_{\lambda\in\Z}(c\tau+d)^{-k}\exp\left(2\ic\pi m\left(\lambda^2\frac{a\tau+b}{c\tau+d}+2\frac{\lambda z-cz^2}{c\tau+d}\right)\right). \]
Let us define %
\[%
\Phi_{10,1}=\textstyle\frac{1}{144}(\E_6\E_{4,1}-\E_4\E_{6,1})\in\wJ_{10,1}, \qquad
\Phi_{12,1}=\textstyle\frac{1}{144}(\E_4^2\E_{4,1}-\E_6\E_{6,1})\in\wJ_{12,1}, %
\]
and %
\[%
\Delta=\frac{1}{1728}(\E_4^3-\E_6^2)\in\Mod_{12}. %
\]
We define in \(\wJ\) the elements 
\begin{equation}\label{def_JacobiAetB}%
\A=\frac{\Phi_{10,1}}{\Delta}\in\wJ_{-2,1}\qquad\text{and}\qquad\B=\frac{\Phi_{12,1}}{\Delta}\in\wJ_{0,1}. %
\end{equation}
By \cite[Theorem 9.3]{MR781735}, we have \[\wJ=\Mod[\A, \B]=\CC[\E_4,\E_6,\A,\B].\]
Using the algorithm proven in~\cite[p. 39]{MR781735}, we can compute the Fourier expansion of \(\Phi_{10,1}\) and \(\Phi_{12,1}\) and deduce the ones of \(\A\) and \(\B\). 
\subsection{A localization of the algebra of the weak Jacobi forms}
We introduce the algebra
\begin{equation}\label{defK}
\wK=\CC[\E_4,\E_6,\A^{\pm 1},\B]\ \supset\wJ.
\end{equation}
This is the localization of \(\wJ\) with respect to the powers of \(\A\). The notions of weight and index naturally extend to \(\wK\) defining a bigraduation%
\begin{equation}\label{defKbigrad}
\wK=\bigoplus_{\substack{k\in 2\Z \\m\in\Z}}\wK_{k,m}.
\end{equation}
We set
\begin{equation}\label{defF2}
\F_2=\B\A^{-1},
\end{equation}
which satisfies
\begin{equation*}
\wK=\CC[\E_4,\E_6,\F_2,\A^{\pm 1}]=\CC[\E_4,\E_6,\F_2][\A^{\pm 1}] 
\end{equation*}
and leads to introduce the subalgebra
\begin{equation*}
\wQ=\CC[\E_4,\E_6, \F_2].
\end{equation*}
The elements of \(\wQ\) appear as the elements in \(\wK\) of index zero. The following table summarizes the weights and indices attached to these different generators.

\begin{center}{\renewcommand{\arraystretch}{1.4}
\begin{tabular}{|c|c|c|c|c|c|}
\hline
\   &   \(\E_4\) &\(\E_6\) &\(\A\) & \(\B\)& \(\F_2\) \\
\hline
weight & 4 & 6 & -2& 0 & 2 \\
\hline
index & 0 &0& 1& 1 & 0\\
\hline
\end{tabular}}
\end{center}

\subsection{Number-theoretic interpretation, relationship with quasimodular forms}
The function \(\F_2\) has a number-theoretic meaning since
\begin{equation}\label{weierstrass}
\F_2=-\frac3{\pi^2}\wp
\end{equation}
where \(\wp\) is the Weierstra\ss{} function \cite[Theorem 3.6]{MR781735} and hence \(\wQ\) is the subalgebra generated by modular forms and the Weierstra\ss{} function 
\begin{equation}\label{eq_extensionWeierstrass}
\wQ=\Mod[\wp].
\end{equation}
Another arithmetical point of view consists in seeing \(\wQ\) as a formal analogue to the algebra \(\QM=\Mod[\E_2]\) of quasimodular forms. The algebra isomorphism involved is 
\begin{equation}\label{defomega}
\omega \colon \wQ \to  \QM, \  P(\E_4,\E_6,\F_2)  \mapsto P(\E_4,\E_6,\E_2).\end{equation}
The degree related to \(\F_2\) of any \(f\in\wQ\) is the depth of the quasimodular form \(\omega(f)\) and we have %
\begin{equation*}
\Mod\subset\QM\simeq\wQ\subset\wK.
\end{equation*}
The isomorphism~\(\omega\) and~\eqref{eq_extensionWeierstrass} emphasize that, from an algebraic point of view, the Weierstra\ss { } function \(\wp\) is similar to the Eisenstein series \(\E_2\).

\subsection{Summary}We can summarize the relationships between the different function algebras under consideration by the following diagram:
\[
\xymatrix{\wJ=\CC[\E_4,\E_6,\A,\B] \ \ar@{^{(}->}[r]                       & \ \wK=\CC[\E_4,\E_6,\F_2,\A^{\pm 1}]     \\
\Mod=\CC[\E_4,\E_6]    \ \ar@{^{(}->}[r] \ar@{^{(}->}[u] &   \ \wQ=\CC[\E_4,\E_6,\F_2]\simeq \QM \ar@{^{(}->}[u]  \\}\ 
\]
This is the framework for our study of formal deformations of \(\wJ\).

\subsection{Problem} Our goal is to construct families of Rankin-Cohen brackets on weak Jacobi forms,
\begin{enumerate}[label=\textnormal{(\roman*)}]
\item which are deformations of the algebra \(\wJ\) (this was not the case for some other construction,
appearing in the literature),
\item which extend Rankin-Cohen brackets on modular forms,  
\item which are coherent with the weight and the index (that is preserve the index and increase the weight by two).
\end{enumerate}
Two main methods can be used following the two points of view illustrated above in the case of modular forms.
\begin{enumerate}[label=\textnormal{(\roman*)}]
\item The first strategy is to start from a derivation of \(\wJ\) and to use the canonical construction of Proposition~\ref{prop_biFRC}. This method gives rise in Section~\ref{firstRC} to a family of deformations of \(\wJ\) extending the Serre-Rankin-Cohen brackets on \(\Mod\) (defined in Section~\ref{SeRCb}).
\item The second one is to apply the extension-restriction process of Theorem~\ref{extrestrict}. So we start from a suitable derivation \(D\) of the extension \(\wK\) of \(\wJ\) which doesn't stabilize \(\wJ\). We thus construct in Section~\ref{secondRC} a family of deformations stabilizing \(\wJ\) and extending the classical Rankin-Cohen brackets on \(\Mod\) (defined in Section~\ref{originalpb}).
\end{enumerate}
\section{A first family of Rankin-Cohen deformations on weak Jacobi forms}\label{firstRC} 
In this section, we define and study a family of Rankin-Cohen deformations on \(\wJ\) whose restriction to \(\Mod\) is the sequence of Serre-Rankin-Cohen brackets. 

\subsection{Construction of the deformations}

We first define an extension of the Serre derivative to the weak Jacobi forms.
For any \(a,b\in\CC\) we denote by \(\Se_{a,b}\) the derivation of \(\wJ=\CC[\E_4,\E_6,\A,\B]\) defined by
\begin{equation}\label{Serabrel}
\Se_{a,b}(\E_4)=-\frac13\E_6, \quad \Se_{a,b}(\E_6)=-\frac12\E_4^2,\quad \Se_{a,b}(\A)=a\B, \quad \Se_{a,b}(\B)=b\E_4\A.\end{equation}
This is by \eqref{Serrel} the only way to extend the Serre derivation \(\Se\) on \(\Mod\) into a derivation of \(\wJ\) preserving the index and increasing the weight by two.
 
We use the general process described in Section~\ref{FRC} to construct formal Rankin-Cohen brackets on \(\wJ\).
\begin{theorem}\label{firstfamily}
For all \((a,b,c)\in\CC^3\), for any \(n\geq 0\), let \(\accol{\cdot,\cdot}^{[a,b,c]}_n\) be the bilinear map from \(\wJ\times\wJ\) to \(\wJ\) defined by bilinear extension of%
\begin{equation}\label{defmu}%
\accol{f,g}^{[a,b,c]}_n=\sum_{r=0}^n(-1)^r\binom{k+cp+n-1}{n-r}\binom{\ell+cq+n-1}{r}\Se_{a,b}^r(f)\Se_{a,b}^{n-r}(g)%
\end{equation}
for all \(f\in\wJ_{k,p}\) and \(g\in\wJ_{\ell,q}\).  Then, 
\begin{enumerate}[label=\textnormal{(\roman*)}]
\item\label{item_eins} The sequence \(\left(\accol{\cdot,\cdot}^{[a,b,c]}_n\right)_{n\geq 0}\) is a formal deformation of \(\wJ\).
\item\label{item_zwei} We have \(\accol{\wJ_{k,p},\wJ_{\ell,q}}^{[a,b,c]}_n\subset\wJ_{k+\ell+2n,p+q}\) for all \(k,\ell\in 2\Z, \ p,q,n\geq 0\).
\item\label{item_drei} The subalgebra of modular forms \(\Mod\) is stable by \(\accol{\cdot,\cdot}^{[a,b,c]}_n\) and its restriction to \(\Mod \times \Mod\) is the Serre-Rankin-Cohen bracket \(\SRC_n\).
\end{enumerate}
\end{theorem}
\begin{proof}
For any \((a,b)\in\CC^2\), the derivation  \(\Se_{a,b}\) of \(\wJ\)  satisfies \(\Se_{a,b}(\wJ_{k,m})\subset\wJ_{k+2,m} \).
Then,~\ref{item_eins} and~\ref{item_zwei} follow from Proposition~\ref{prop_biFRC} since \(\accol{\cdot,\cdot}^{[a,b,c]}_n=\FRC_n^{\Se_{a,b},c}\).
If \(f,g\) are modular forms of respective weights \(k,\ell\), we have \(p=q=0\) in formula \eqref{defmu} which doesn't depend on \(c\) in this case.    
Moreover \(\Se_{a,b}(f)=\Se(f)\) and \(\Se_{a,b}(g)=\Se(g)\) by \eqref{Serabrel} and \eqref{Serrel}. Hence \(\accol{f,g}^{[a,b,c]}_n\) doesn't depend on \((a,b)\)
and is by \eqref{SRCFM} equal to \(\SRC_n(f,g)\).\end{proof}

\subsection{Classification and separation results}\label{classifsep}

The formal deformations \((\accol{\cdot,\cdot}_n^{[a,b,c]})_{n\geq 0}\) depend on three parameters. We can classify them, up to a suitable specialization of the notion of isomorphic deformations (see paragraph \ref{equivfd}) with respect to the arithmetical context.

\begin{definition}\label{modiso}
Two formal deformations \(\left(\accol{\cdot,\cdot}_n^{[a,b,c]}\right)_{n\geq 0}\) and \(\left(\accol{\cdot,\cdot}_n^{[a',b',c']}\right)_{n\geq 0}\) of \(\wJ\) 
are modular-isomorphic if there exists a \(\CC\)-linear bijective map \(\phi\colon\wJ\to\wJ\) such that %
\begin{enumerate}[label=\textnormal{(\roman*)}]
\item\label{item_moddeux} \(\phi(\accol{f,g}_j^{[a,b,c]})=\accol{\phi(f),\phi(g)}_j^{[a',b',c']}\) for all \(j\geq 0\) and \(f,g\in\wJ\).
\item\label{item_modun} \(\phi\) preserves the index and the weight of homogeneous weak Jacobi forms %
\end{enumerate}
\end{definition}
In particular \(\phi\) is a \(\CC\)-algebra automorphism of \(\wJ\) and a Poisson isomorphism from \((\wJ,\accol{\cdot,\cdot}_1^{[a,b,c]})\) to \((\wJ,\accol{\cdot,\cdot}_1^{[a',b',c']})\).

\begin{lemma}\label{CNseparation}
If two formal deformations \((\accol{\cdot,\cdot}_n^{[a,b,c]})_{n\geq 0}\) and \((\accol{\cdot,\cdot}_n^{[a',b',c']})_{n\geq 0}\) are modular-isomorphic, then \(c=c'\), and there exists \(\xi\in\CC^*\) such that \(a'=\xi a\) and \(b'=\xi^{-1}b\).
\end{lemma}
\begin{proof}
Let \(\phi\colon\wJ\to\wJ\) be as in Definition~\ref{modiso}. It induces a Poisson isomorphism from \(\left(\wJ,\accol{\cdot,\cdot}_1^{[a,b,c]})\right)\) to \(\left(\wJ,\accol{\cdot,\cdot}_1^{[a',b',c']}\right)\). We have \(\phi\left(\Mod\right)\subset\Mod\) by~\ref{item_modun} and \(\RC_1=\SRC_1=\accol{\cdot,\cdot}_1^{[a,b,c]}\) on \(\Mod\). Then, the restriction of \(\phi\) to \(\Mod\) is a Poisson modular automorphism of \(\left(\Mod,\RC_1\right)\). By~\cite[Proposition 7]{DR}, this is the identity of \(\Mod\).

Let \(f\in\wJ_{k,p}\). Then, \(\phi(f)\in\wJ_{k,p}\). The restriction of \(\Se_{a,b}\) to \(\Mod\) is \(\Se\). The kernel of \(\Se\) is \(\CC[\Delta]\) (see, for example, \cite[Proposition 8]{DR}) and \(\phi(\Delta)=\Delta\). We deduce that %
\[%
\phi\left(\accol{f,\Delta}_1^{[a,b,c]}\right)=-12\Delta\phi\left(\Se_{a,b}(f)\right)\quad \text{ and }\quad \accol{\phi(f),\phi(\Delta)}_1^{[a',b',c']}=-12\Delta\Se_{a',b'}\left(\phi(f)\right) %
\]%
and hence %
\begin{equation}\label{commute}
\phi\circ\Se_{a,b}=\Se_{a',b'}\circ\phi. 
\end{equation}
It follows that, for all \(f\in\wJ_{k,p}\) and \(g\in\wJ_{\ell,q}\), we have%
\[%
\accol{\phi(f),\phi(g)}_1^{[a',b',c']}=\phi\left((k+c'p)f\Se_{a,b}(g)-(\ell+c'q)g\Se_{a,b}(f)\right)
\]
and~\ref{item_moddeux} leads to%
\[%
(c'-c)\left(pf\Se_{a,b}(g)-qg\Se_{a,b}(f)\right)=0.
\]
We apply this equality to \(f=\A\E_4\) and \(g=\E_6\) to obtain \(c'=c\). Let \(\lambda\) and \(\mu\) be defined in \(\CC^*\) by \(\phi(\A)=\lambda\A\) and \(\phi(\B)=\mu\B\). Equation~\eqref{commute} applied to \(\A\) and \(\B\) gives \(a\mu=a'\lambda\) and \(b'\mu=b\lambda\). The proof is complete setting \(\xi=\mu/\lambda\). %
\end{proof}

\begin{theorem}\label{classification}
Let \((a,b,c)\in\CC^3\). The formal deformation \((\accol{\cdot,\cdot}_n^{[a,b,c]})_{n\geq 0}\) of \(\wJ\) is modular-isomorphic to one of the following formal deformations, %
\begin{enumerate}[label=\textnormal{(\arabic*)}]
\item The formal deformation \((\accol{\cdot,\cdot}_n^{[1,b',c']})_{n\geq 0}\) for some \((b',c')\in\CC^2\). 
\item The formal deformation \((\accol{\cdot,\cdot}_n^{[0,1,c']})_{n\geq 0}\) for some \(c'\in\CC\).
\item The formal deformation \((\accol{\cdot,\cdot}_n^{[0,0,c']})_{n\geq 0}\) for some \(c'\in\CC\).
\end{enumerate}
These deformations are pairwise non modular-isomorphic for different values of the parameters. %
\end{theorem}
\begin{proof}
For any \((\lambda,\mu)\in\CC^{*2}\), denote by \(\phi_{\lambda,\mu}\) the \(\CC\)-algebra automorphism of \(\wJ\) fixing \(\E_4\) and \(\E_6\) and 
such that \(\phi_{\lambda,\mu}(\A)=\lambda\A\) and \(\phi_{\lambda,\mu}(\B)=\mu\B\). By the same calculations as in proof of lemma \ref{CNseparation} we have, for any \((a,b,a',b')\in\CC^4\),%
\begin{equation}\label{CSiso}
(\phi_{\lambda,\mu}\circ\Se_{a,b}=\Se_{a',b'}\circ\phi_{\lambda,\mu})\ \text{ if and only if } (a\mu=a'\lambda \ \text{and } b\lambda=b'\mu). 
\end{equation}
By \eqref{defmu}, if one of the equivalent conditions of \eqref{CSiso} is satisfied, then the formal deformations \(\accol{\cdot,\cdot}^{[a,b,c]}\) and \(\accol{\cdot,\cdot}^{[a',b',c']}\) are isomorphic.  Since it follows from~\eqref{CSiso} that %
\(\phi_{a,1}\circ\Se_{a,b}=\Se_{1,ab}\circ\phi_{a,1}\) for any \(a\not=0\), and \(\phi_{1,b}\circ\Se_{0,b}=\Se_{0,1}\circ\phi_{1,b}\) for any \(b\not=0,\)
the proof that \((\accol{\cdot,\cdot}_n^{[a,b,c]})_{n\geq 0}\) is modular-isomorphic to one of given formal deformations is complete. The separation of the different cases up to modular isomorphism follows from a direct application of Lemma~\ref{CNseparation}.
\end{proof}

\subsection{Oberdieck's derivation and a new way to build Jacobi forms}

The extended Serre derivation \(\Se_{a,b}\) is a generalization of an already known derivation of the algebra of weak Jacobi forms that preserves the index, increases the weight by \(2\), and has an analytic expression: the Oberdieck derivation~\cite{Ober12}. For the convenience of the reader, we only briefly describe this derivation here. The interested reader will find details in the unpublished note~\cite{CDMROber}.

For all \(\tau\in\pk\), let \(\Lambda_\tau=\Z\oplus\tau\Z\). The \(\zeta\) function associated to \(\Lambda_\tau\) is defined by %
\begin{equation}\label{eq_zetaWeier}%
\forall z\in\CC-\Lambda_\tau\qquad\zeta(\tau,z)=\frac{1}{z}+\sum_{\substack{\omega\in\Lambda_\tau\\\omega\neq0}}\left(\frac{1}{z-\omega}+\frac{1}{\omega}+\frac{z}{\omega^2}\right).
\end{equation}

Let \(\J_1\) and \(\J_2\) be the two functions defined by%
\[%
\forall\tau\in\pk,\,\forall z\in\CC,\, z\notin\Z+\tau\Z\qquad\J_1(\tau,z)=\frac{1}{2\pi\ic}\zeta(\tau,z)+\frac{\pi\ic}{6}z\E_2(\tau),
\]
and %
\[%
\J_2=\frac{1}{2\pi\ic}\partial_z\J_1-\frac{1}{12}\E_2+\J_1^2%
\]
where \(\partial_z\) is the derivative with respect to the second variable.

We define an application on \(\wJ\) by%
\[%
\Ob(f)=\frac{1}{2\pi\ic}\partial_\tau(f)-\frac{k}{12}\E_2 f-\frac{1}{2\pi\ic}\J_1\partial_z(f)+m\J_2f %
\]%
for any \(f\) in the space \(\wJ_{k,m}\). It can be shown that this application is a derivation of \(\wJ\) satisfying \(\Ob\left(\wJ_{k,m}\right)\subset\wJ_{k+2,m}\). Computing the values of \(\Ob\) at \(\E_4\), \(\E_6\), \(\A\) and \(\B\), we see that \(\Ob=\Se_{-1/6,-1/3}\). %

The existence of an analytic expression for \(\Ob\) provides an explicit way to build new Jacobi forms from old ones. For examples, \(\E_6\) is obtained from \(\E_4\) by \(\Ob\) and \(\B\) is obtained from \(\A\) by \(\Ob\).

\subsection{Relationship with deformations on quasimodular forms}
The deformations we have built on the algebra of weak Jacobi forms  produce deformations on the algebra of quasimodular forms. 
More precisely they extend from \(\wJ\) to \(\wK\) and then, restrict  to \(\wQ=\CC[\E_4,\E_6,\F_2]\). 
For any \(a,b\in\CC\), the derivation \(\Se_{a,b}\) extends canonically to \(\wK\) by%
\[%
\Se_{a,b}(\A^{-1})=-\A^{-2}\Se_{a,b}(\A)=-a\A^{-2}\B.%
\]
This implies by \eqref{defF2}
\begin{equation*}\label{serreF2}
\Se_{a,b}(\F_2)=b\E_4-a\F_2^2.
\end{equation*}
It follows that the algebra \(\wQ\) is stable by \(\Se_{a,b}\). Hence \(\accol{f,g}_n^{[a,b,c]}\in\wQ\) for all \(n\geq 0\) and \(f,g\in\wQ\).
These expressions do not depend on \(c\) since the functions of \(\wQ\) have index \(0\). We denote simply 
\begin{equation}\label{defmuquasi}%
\accol{f,g}^{[a,b]}_n=\sum_{r=0}^n(-1)^r\binom{k+n-1}{n-r}\binom{\ell+n-1}{r}\Se_{a,b}^r(f)\Se_{a,b}^{n-r}(g)%
\end{equation}
for all homogeneous \(f,g\in\wQ\) of respective weights \(k\) and \(\ell\).  This defines a deformation of \(\wQ\).
We consider three cases, %
\begin{enumerate}[label=\textnormal{(\arabic*)}]
\item If \(a=b=0\), the deformation \((\accol{f,g}^{[0,0]}_n)_{n\geq 0}\) of \(\wQ\) corresponds up the isomorphism \(\omega\) defined in \eqref{defomega} between \(\wQ\) and \(\QM\) to the deformation of \(\QM\) studied in \cite[Theorem D]{DR} for the particular case \(b=0\) and \(\alpha=-\frac23\) of the parameters used in this theorem.
\item If \(a=0\) and \(b\neq 0\),  we can by Theorem~\ref{classification} reduce to the deformation \((\accol{f,g}^{[0,-1/12]}_n)_{n\geq 0}\) of \(\wQ\) which corresponds up the isomorphism \(\omega\) to the deformation of \(\QM\) studied in \cite[Theorem B]{DR} for the particular case \(a=0\) of the parameter used in this theorem.
\item\label{item_Quut} If \(a\neq 0\), the deformation of \(\QM\) obtained from \((\accol{f,g}^{[a,b]}_n)_{n\geq 0}\) through the isomorphism \(\omega\) does not correspond to any bracket defined in~\cite{DR}.  The reason is that the study of \cite{DR} was devoted to deformations preserving the depth of the quasimodular forms. This is not the case here since for instance
\[
\accol{\E_4,\F_2}_1^{[a,b]}=4b\E_4^2+\frac{2}{3}\E_6\F_2-4a\E_4\F_2^2
\]
is of \(\F_2\)-degree two whereas \(\E_4\) and \(\E_2\) are of respective depth 0 and 1 in \(\QM\).
\end{enumerate}


\section{A second family of Rankin-Cohen deformations on weak Jacobi forms}\label{secondRC}
In this section, we define and study a family  of Rankin-Cohen deformations on \(\wJ\)  whose restriction to \(\Mod\) is the sequence of classical Rankin-Cohen brackets. The method consists in applying the extension-restriction principle described in Paragraph \ref{ERMethod} to some family of  formal deformations of \(\wK\).
\subsection{Construction of the deformations}
Recall that \(\omega\colon\wQ\to\QM\) is the algebra isomorphism that sends \((\E_4,\E_6,\F_2)\) to \((\E_4,\E_6,\E_2)\). The usual complex derivative \(\Dz\) defines a derivation on the algebra \(\QM\) of quasimodular forms. We define a derivation on \(\wQ\) by \(\partial=\omega^{-1}\Dz\omega\). Ramanujan equations~\eqref{Ramanujan} become%
\begin{equation}\label{defpartialu}%
\partial(\E_4)=-\frac{1}{3}(\E_6-\E_4\F_2),\quad\partial(\E_6)=-\frac{1}{2}(\E_4^2-\E_6\F_2),\quad\partial(\F_2)=-\frac{1}{12}(\E_4-\F_2^2).
\end{equation}
The unique way to extend \(\partial\) into a  derivation \(\partial_u\) of \(\wK\) preserving the index and increasing the weight by 2 is to set
\begin{equation}\label{defpartialu2}
\forall f\in\wQ\quad\partial_u(f)=\partial(f) \text{ and } \partial_u(\A)={u}\A\F_2
\end{equation}
for some \(u\in\CC\). We compute
\begin{equation}\label{defpartialu3}
\partial_u(\B)=\partial_u(\A\F_2)=(u+\frac1{12})\B\F_2-\frac1{12}\E_4\A.
\end{equation}
For any \(u\in\CC\), the derivation \(\partial_u\) does not restrict into a derivation of \(\Mod\) nor into a derivation of \(\wJ\).

Applying the process described Section~\ref{FRC}, we define, using \(\partial_u\), the following deformation of \(\wK\). 

\begin{theorem}\label{thm3}
For any complex parameters \(u\) and \(v\), let \(\left(\CCrochet{\cdot,\cdot}^{u,v}_n\right)_{n\geq 0}\) be the sequence of maps \(\wK\times\wK\to\wK\) defined by bilinear extension of%
\begin{equation}\label{defnu}
\CCrochet{f,g}^{u,v}_n=\sum_{r=0}^n(-1)^r\binom{k+vp+n-1}{n-r}\binom{\ell+vq+n-1}{r}\partial_u^r(f)\partial_u^{n-r}(g), %
\end{equation}
for all homogeneous \(f\in\wK_{k,p}\) and \(g\in\wK_{\ell,q}\). Then, for all \((u,v)\in\CC^2\),%
\begin{enumerate}[label=\textnormal{(\roman*)}]
\item\label{item_uks} The sequence \(\left(\CCrochet{\cdot,\cdot}^{u,v}_n\right)_{n\geq 0}\) is a formal deformation of \(\wK\).
\item\label{item_kaks} \(\CCrochet{\wK_{k,p},\wK_{\ell,q}}_n^{u,v}\subset\wK_{k+\ell+2n,p+q}\).
\item\label{item_kolm} The subalgebra \(\Mod\) of modular forms is stable by \(\CCrochet{\cdot,\cdot}^{u,v}_n\), its restriction to \(\Mod\) being the classical Rankin-Cohen bracket  \(\RC_n\).
\end{enumerate}
\end{theorem}
\begin{proof}
For any \(u\in\CC\), the derivation \(\partial_u\) of \(\wK\) satisfies \(\partial_u\left(\wK_{k,m}\right)\subset\wK_{k+2,m}\). Then,~\ref{item_uks} and~\ref{item_kaks} follow from Proposition~\ref{prop_biFRC} since \(\CCrochet{\cdot,\cdot}^{u,v}_n=\FRC_n^{\partial_u,v}\). If \(f,g\) are modular forms of respective weights \(k,\ell\), we have \(p=q=0\) in formula \eqref{defnu} which doesn't depend on \(v\) in this case.    
Moreover \(\partial_u(f)=\partial(f)=D_z(f)\) and \(\partial_u(g)=\partial(g)=D_z(g)\). Hence \(\CCrochet{f,g}^{[u,v]}_n\) doesn't depend on \((u,v)\)
and is by \eqref{classicalRC} equal to \(\RC_n(f,g)\), which proves~\ref{item_kolm}.\end{proof}

\begin{remark}[classification and separation result]
Using arguments similar to those in \ref{classifsep}, we can show that two formal deformations \((\CCrochet{\cdot,\cdot}_n^{u,v})_{n\geq 0}\) and \((\CCrochet{\cdot,\cdot}_n^{u',v'})_{n\geq 0}\) of \(\wK\) are modular isomorphic if and only if \((u,v)=(u',v')\). The details of the proof are left to the reader.
\end{remark}

\begin{remark}[relationship with deformations on quasimodular forms]\label{rem_vc}
The subalgebra \(\wQ\) is stable by the brackets \(\CCrochet{\cdot,\cdot}^{u,v}_n\). However, their restrictions  to \(\wQ\) do not preserve the degree in \(\F_2\)  (so up to the isomorphism \(\omega\), they do not preserve the depth of quasimodular forms). For this reason, the restrictions of  \(\CCrochet{\cdot,\cdot}^{u,v}_n\) to the subalgebra \(\wQ\) can not coincide with the brackets previously studied in \cite{DR}.
\end{remark}


\subsection{Restriction to weak Jacobi forms}
Although the subalgebra \(\wJ\) is not stable by the derivation \(\partial_u\), the question arises whether \(\wJ\) can be stable by \(\left(\CCrochet{\cdot,\cdot}^{u,v}_n\right)_{n\geq 0}\) for some values of the parameters \(u\) and \(v\). Since
\[
\CCrochet{\B,\E_4}_1^{u,v}=\frac13\left(v-(12u+1)\right)\B\E_4\F_2-\frac13v\E_6\B+\frac13\E_4^2\A,
\]
a necessary condition is \(v=12u+1\). We use the general method of Theorem~\ref{extrestrict} to prove that this condition is sufficient.

\begin{theorem}\label{thm4}
For any complex number \(u\), the sequence \(\left(\CCrochet{\cdot,\cdot}^{u,12u+1}_n\right)_{n\geq 0}\) defines by restriction a deformation of the algebra \(\wJ\) of weak Jacobi forms, whose restriction to the subalgebra \(\Mod\) of modular forms is the sequence of classical Rankin-Cohen brackets. %
\end{theorem}

\begin{proof}
We consider the derivation \(\Delta_u\) of \(\wJ\) defined by \(\Delta_u(f)=\frac12(k+(12u+1)p)f\) for any weak Jacobi form \(f\) of weight \(k\) and index \(p\).
We have
\[\Delta_u(\E_4)=2\E_4, \quad\Delta_u(\E_6)=3\E_6,\quad \Delta_u(\A)=\left(6u-\frac12\right)\A, \quad  \Delta_u(\B)=\left(6u+\frac12\right)\B.\]
We also denote by \(\Delta_u\) its canonical extension as a derivation of \(\wK\), which satisfies:
\[\Delta_u(\F_2)=\Delta_u(\B)\A^{-1}-\B\Delta_u(\A)\A^{-2}=\F_2.\] 
We denote \(\theta=\Se_{\frac 1{12}, -\frac 1{12}}\) in the sense of \eqref{Serabrel}. So we have
\[\theta(\E_4)=-\frac13\E_6,\quad \theta(\E_6)=-\frac12\E_4^2,\quad \theta(\A)=\frac1{12}\B,\quad \theta(\B)=-\frac1{12}\E_4\A,\]
and its canonical extension as a derivation of \(\wK\) satisfies:
\[\theta(\F_2)=\theta(\B)\A^{-1}-\B\theta(\A)\A^{-2}=-\frac1{12}\F_2^2-\frac1{12}\E_4.\]
The derivations \(\Delta_u\) and \(\theta\) of \(\wK\) satisfy \(\Delta_u\theta-\theta\Delta_u=\theta,\) and by construction the subalgebra \(\wJ\) of \(\wK\) is stable par \(\Delta_u\) et \(\theta\).
The element \(h=-\frac1{144}\E_4\) of \(\wJ\) satisfies \(\Delta_u(h)=2h\). Then, the element \(x=\frac1{12}\F_2\) of \(\wK\) satisfies
\(x\notin\wJ\), \(\Delta_u(x)=x\) and \(\theta(x)=-x^2+h\). Thus we can apply Theorem~\ref{extrestrict}.
The derivation \(D:=\theta+2x\Delta_u\) of \(\wK\) is such that \(\Delta_u D-D\Delta_u=D\) and
\(\wJ\) is stable by the Connes-Moscovici deformation \((\CCM_n^{D,\Delta_u})_{n\geq 0}\) of \(\wK\).

On the one hand, \(D(\wK_{k, p}) \subset \wK_{k+2, p}\). Then \((\CCM_n^{D,\Delta_u})_{n\geq 0} = (\FRC_n^{D,12u+1})_{n\geq 0}\) on \(\wK\) by Proposition~\ref{prop_biFRC}. Hence we have for any \(f\in \wJ_{k,p}\) and \(g\in\wJ_{\ell,q}\),
\begin{multline*}%
\CCM_n^{D,\Delta_u}(f,g)=\\\sum_{r=0}^n(-1)^r\binom{k+(12u+1)p+n-1}{n-r}\binom{\ell+(12u+1)q+n-1}{r}D^r(f)D^{n-r}(g).%
\end{multline*}
On the other hand, the calculation of the values of \(D\) on the generators of \(\wK\) gives
\[
D(\E_4)=\frac13(\E_4\F_2-\E_6),\, D(\E_6)=\frac12(\E_6\F_2-\E_4^2),\, D(\A)=u\A\F_2,\, D(\F_2)=\frac1{12}(\F_2^2-\E_4).
\]
We deduce by \eqref{defpartialu} and \eqref{defpartialu2} that \(D\) is equal to the derivation \(\partial_u\) 
and then, by \eqref{defnu} that
\[%
(\CCM_n^{D,\Delta_u})_n=([[{\cdot,\cdot}]]^{u,12u+1}_n)_n.
\]
We conclude that \(\wJ\) is stable by the formal deformation \(([[{\cdot,\cdot}]]^{u,12u+1}_n)_n\) of \(\wK\).

The fact that the restriction to \(\Mod\) of this deformation of \(\wJ\) is the sequence of classical Rankin-Cohen brackets follows from~\ref{item_kolm} of Theorem~\ref{thm3}.
\end{proof}
\bibliographystyle{plain}
\bibliography{ExtNewJacobi}
\end{document}